\documentclass[letterpaper,10pt,conference]{ieeeconf}

\IEEEoverridecommandlockouts
\ifdefined\pdfminorversion
\fi

\usepackage{cite}
\usepackage{amsmath,amssymb,amsfonts}
\usepackage{graphicx}
\usepackage{booktabs}
\usepackage{algorithmic}
\usepackage[hidelinks]{hyperref}

\newtheorem{proposition}{Proposition}

\newcounter{dfalgorithm}
\renewcommand{\thedfalgorithm}{\arabic{dfalgorithm}}

\graphicspath{{figures/}}

\title{\LARGE \bf
RapidDOC: Riccati-Preconditioned Implicit Differentiation for
Differentiable Optimal Control under Operator Mismatch
}

\author{Xiangda Zhu, Yuankun Chen, Zifei Nie$^{*}$, and Yi Chang$^{*}$%
\thanks{Xiangda Zhu, Yuankun Chen, Zifei Nie, and Yi Chang are with the School of Artificial Intelligence, Jilin University, China.\newline
E-mails: \texttt{\{}\href{mailto:zhuxd25@mails.jlu.edu.cn}{\nolinkurl{zhuxd25}}, \href{mailto:chenyk24@mails.jlu.edu.cn}{\nolinkurl{chenyk24}}\texttt{\}@mails.jlu.edu.cn};\newline
\texttt{\{}\href{mailto:zifei_nie@jlu.edu.cn}{\nolinkurl{zifei_nie}}, \href{mailto:yichang@jlu.edu.cn}{\nolinkurl{yichang}}\texttt{\}@jlu.edu.cn}.}%
\thanks{$^{*}$Corresponding authors: Zifei Nie and Yi Chang.}}

\begin{document}

\bstctlcite{BSTcontrol}
\maketitle
\thispagestyle{empty}
\pagestyle{empty}

\begin{abstract}
Differentiable optimal control embeds constrained optimization problems in data-driven pipelines by differentiating their optimality conditions and propagating task-loss gradients to model or objective parameters.
A key observation is that forward optimization and backward sensitivity computation involve related stage-structured linear systems, making reuse of forward factorization information natural. However, direct reuse is generally invalid, because differences in linearization, barrier terms, regularization, and curvature can create a mismatch between the retained forward operator and the target operator defining the derivative.
We therefore propose RapidDOC, a factorization-reuse method that leaves
the target adjoint equation unchanged. The retained cache produces a candidate
adjoint and, when its target residual exceeds the prescribed tolerance,
serves as a fixed right preconditioner for GMRES. A cache-action error
operator characterizes directional exactness, the preconditioned
iteration, and the relation between target residual and outer-gradient
error.
An ablation on three constrained robotic OCPs shows that the proposed method achieves lower backward time than both target-system refactorization and a no-reuse iterative method, demonstrating the computational benefit of reusing the retained forward factorizations.
Across open- and closed-loop parameter-learning experiments, the
proposed method achieves learning performance comparable to the
baselines, while yielding the lowest open-loop backward time on all
three tasks and reducing total closed-loop solver-call time by
$40.5$--$83.1\%$ relative to the fastest competing method.
Finally, deploying the learned controller on a physical vehicle reduces the median figure-eight tracking RMSE from $22.75$ to $6.09$~cm, demonstrating its potential for hardware deployment.
\end{abstract}

\begin{keywords}
Differentiable optimization, factorization reuse, implicit differentiation,  mobile robotics, imitation learning.
\end{keywords}
\section{Introduction}

Robotic systems often adapt model parameters, objective weights, or
reference trajectories as tasks and operating conditions change.
Differentiable optimal control enables such adaptation by embedding a
constrained OCP in a learning pipeline
\cite{amos2017optnet,amos2018diffmpc,jin2020pdp}. Gradients propagated
through the optimized trajectory support system identification,
inverse optimal control, learning from demonstrations, and data-driven
model predictive control \cite{jin2020pdp,jin2021safepdp}. Their cost, however, includes both the nominal computation performed
by the forward solver and a backward sensitivity calculation.

The backward pass creates a specific tension. A structure-exploiting
solver has already computed expensive stage-structured Newton or KKT
factorization information, so reusing it is numerically attractive.
Yet the operator represented by that information need not be the one
that defines the desired derivative. Stale linearizations, barrier
updates, regularization, and curvature approximations can all separate
the cached forward operator from the derivative-defining target
operator. Treating the cache as an exact target factorization can then
satisfy a different adjoint equation. The issue is therefore not factor
reuse alone, but reuse without redefining the derivative.

% Unnumbered resource note in the first page's right column.
\begingroup
\renewcommand{\thefootnote}{}
\footnotetext{Code: \url{https://github.com/opt-man/RapidDOC}.\newline
Supplementary video: \url{https://youtu.be/b53gcrjdS7c}.}
\endgroup

Sparse nonlinear-programming and structure-exploiting solvers such as
IPOPT, acados, HPIPM, and FATROP accelerate nominal optimization using
sparse factorizations, condensing, or Riccati recursions
\cite{wachter2006ipopt,verschueren2022acados,frison2020hpipm,
vanroye2023fatrop,vanroye2024generalization}. Differentiable
optimization methods compute sensitivities using iterative implicit
differentiation, auxiliary optimal control systems, or specialized
sensitivity recursions
\cite{agrawal2019diffcp,jin2020pdp,jin2021safepdp,
chen2025fastdoc,chen2026panda}. Direct factorization reuse is possible when
the forward and backward systems are identical or deliberately aligned,
as in OptNet and differentiable MPC
\cite{amos2017optnet,amos2018diffmpc}. sIPOPT reuses IPOPT
factorizations for parametric sensitivity analysis
\cite{pirnay2012sipopt}, while DiffMPC shares an approximate KKT
representation across forward and backward Krylov solvers
\cite{adabag2026diffmpc}. Conversely, dNMPC constructs the prescribed
backward system separately \cite{frey2025dnmpc}. These methods either align the forward and backward operators for
direct reuse or construct the prescribed backward system independently.

We address the intermediate regime in which the retained forward
factorization remains useful despite a mismatch with the
derivative-defining operator. To address this gap, we propose \textbf{RapidDOC}, short for
\textbf{R}icc\textbf{A}ti-\textbf{P}reconditioned
\textbf{I}mplicit \textbf{D}ifferentiation for
\textbf{D}ifferentiable \textbf{O}ptimal \textbf{C}ontrol
under operator mismatch.
In the \textbf{forward pass}, the OCP solver retains its last available
generalized Riccati factorization. In the \textbf{backward pass},
\textbf{RapidDOC} uses these factors to obtain an adjoint candidate and, when
needed, to right-precondition GMRES for the prescribed target adjoint
system. The target equation remains unchanged, and its residual
determines whether the computed adjoint is accepted.

\begin{figure}[!t]
\centering
\vspace{6mm}
\includegraphics[width=0.48\textwidth]{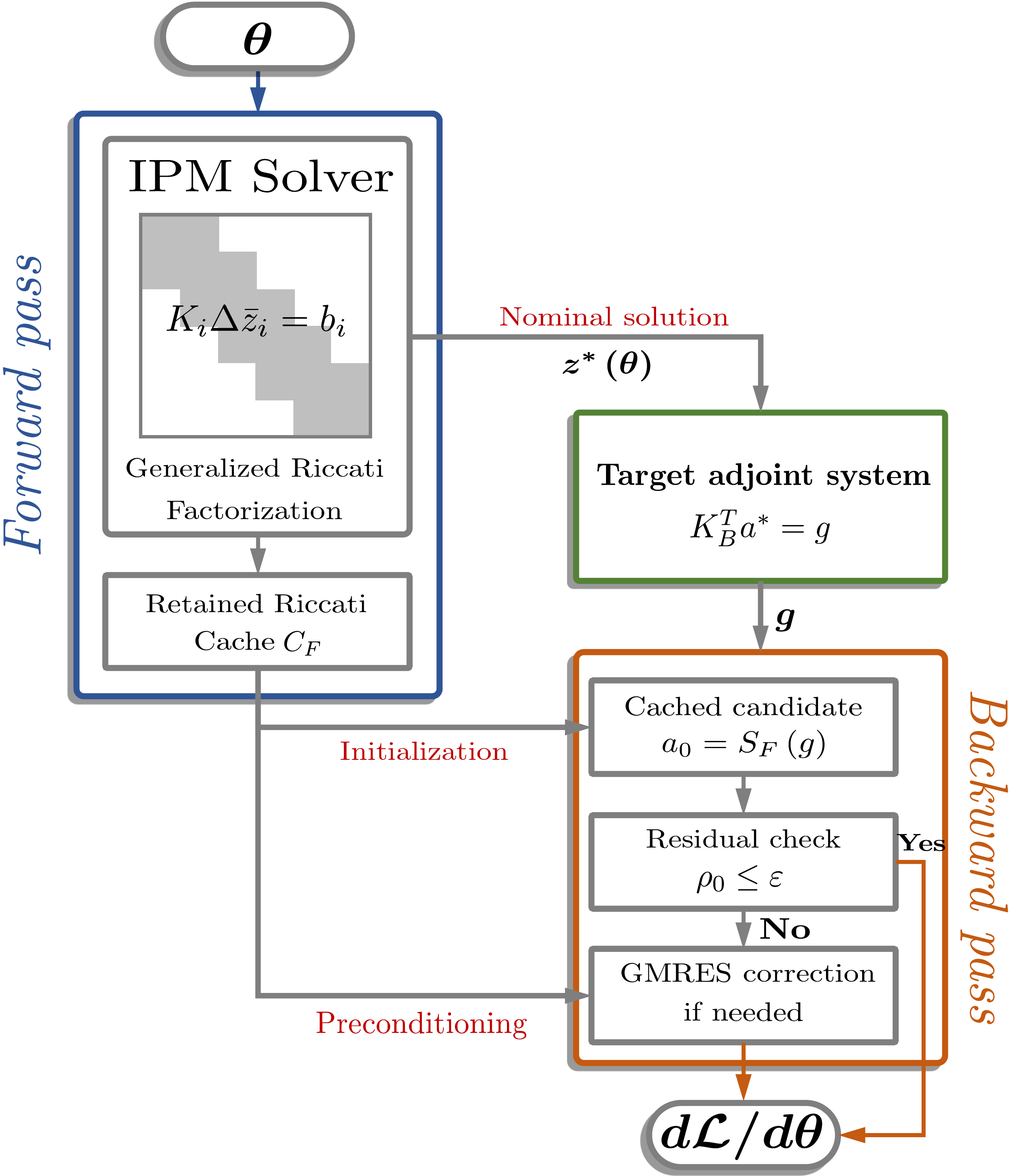}
\caption{Proposed factorization-reuse framework. The retained Riccati cache
provides an adjoint candidate and right preconditioner, while the target
residual determines acceptance.}
\label{fig:framework}
\end{figure}

Figure~\ref{fig:framework} illustrates this forward--backward
interface. The main contributions are:

\begin{itemize}

\item We develop a factorization-reuse backward scheme for
forward--backward operator mismatch. The prescribed target operator
remains the sole definition of the derivative, while the retained
generalized Riccati cache provides an adjoint candidate and a fixed
right preconditioner. This permits directional direct reuse when valid
and cache-guided correction otherwise, without requiring a fresh
target-system factorization.

\item We introduce a cache-action error operator that provides a unified
analysis of direct reuse and cache-guided correction. It characterizes
when reuse is exact for a given loss direction, governs the
right-preconditioned correction, and connects the certified target
residual to outer-gradient error. This analysis shows that reuse depends
on the mismatch relevant to the current loss direction rather than
global equality of the forward and backward operators.

\item We validate the method through a three-way factor-reuse ablation,
open- and closed-loop parameter learning on three constrained robotic
OCPs, and physical-vehicle deployment. The ablation compares reuse with
target-system refactorization and a zero-start, unpreconditioned
no-reuse baseline; the learning experiments demonstrate consistent
learning behavior with lower backward and repeated-MPC computation
time; and physical figure-eight tracking demonstrates hardware
deployment potential.

\end{itemize}
\section{Problem Formulation}
\label{sec:problem}

\subsection{Parameterized Optimal Control}

We consider a parameterized finite-horizon nonlinear OCP. At stage
$k$, let $x_k\in\mathbb{R}^{n_x}$ and
$u_k\in\mathbb{R}^{n_u}$ denote the state and control input. The
parameter vector $\theta\in\mathbb{R}^{n_\theta}$ may enter the
dynamics, initial condition, objective, or constraints. For a horizon
of $N$ stages, define
$\xi:=\operatorname{col}(x_0,u_0,\ldots,x_{N-1},u_{N-1},x_N)$.
For a given $\theta$, the OCP is, with stage-indexed constraints
imposed for $k=0,\ldots,N-1$,
\begin{equation}
\begin{aligned}
\min_{\xi}\quad
& J(\xi,\theta):=
\sum_{k=0}^{N-1}\ell_k(x_k,u_k,\theta)
+\ell_N(x_N,\theta)\\
\mathrm{s.t.}\quad
& x_0=\bar{x}(\theta),\\
& x_{k+1}=f_k(x_k,u_k,\theta),
\\
& c_k(x_k,u_k,\theta)=0,\quad
  h_k(x_k,u_k,\theta)\leq0,
\\
& c_N(x_N,\theta)=0,\quad
  h_N(x_N,\theta)\leq0.
\end{aligned}
\label{eq:ocp}
\end{equation}
Here, $\ell_k$ and $\ell_N$ are the stage and terminal costs,
$f_k$ denotes the discrete-time dynamics, and $\bar{x}(\theta)$
specifies the initial condition. The functions $c_k$ and $c_N$
represent equality constraints, whereas $h_k$ and $h_N$ represent
componentwise inequality constraints. These functions are assumed to
be twice continuously differentiable on the relevant domain.

To obtain a smooth target for differentiation, we use a fixed-barrier
reformulation of the OCP. For a barrier coefficient $\tau>0$, held
fixed during differentiation, define
\begin{equation}
\begin{aligned}
\widetilde J(\xi,\theta)
:=\;&J(\xi,\theta)
-\tau\sum_{k=0}^{N-1}
\mathbf{1}^{\top}
\log\!\bigl(-h_k(x_k,u_k,\theta)\bigr)
\\
&-\tau\mathbf{1}^{\top}
\log\!\bigl(-h_N(x_N,\theta)\bigr).
\end{aligned}
\label{eq:fixed_barrier_objective}
\end{equation}
The logarithm is applied componentwise on the strictly feasible
domain. For fixed $\tau$, let $\xi^\star(\theta)$ denote a local
solution of the resulting equality-constrained barrier problem. All
derivatives below refer to this finite-$\tau$ solution map rather than
to the active-set sensitivity of the original inequality-constrained
OCP.

\subsection{Adjoint System}

Let $q(\xi,\theta)$ collect the initial-state, dynamics, and equality
constraint residuals in \eqref{eq:ocp}, and let $\lambda$ denote their
multipliers. With $z:=\operatorname{col}(\xi,\lambda)$, the first-order
optimality residual of the fixed-barrier problem is
\begin{equation}
\mathcal F(z,\theta):=
\begin{bmatrix}
\nabla_\xi\widetilde J(\xi,\theta)
+
\left(
\dfrac{\partial q(\xi,\theta)}{\partial\xi}
\right)^\top\lambda\\[3pt]
q(\xi,\theta)
\end{bmatrix}.
\label{eq:optimality_residual}
\end{equation}

Let $\lambda^\star(\theta)$ denote the multipliers associated with
$\xi^\star(\theta)$. The selected primal--dual solution
$z^\star(\theta):=\operatorname{col}(\xi^\star(\theta),
\lambda^\star(\theta))$ satisfies
$\mathcal F(z^\star(\theta),\theta)=0$.
We assume that $\mathcal F$ is continuously differentiable near this
root and that its Jacobian with respect to $z$ is nonsingular.
Throughout, the nominal primal--dual solution is required to satisfy
these prescribed optimality conditions to the stated forward tolerance.
The mismatch studied below concerns the retained linear operator, not a
change in the differentiation point.

At the selected solution, define the prescribed backward operators
\begin{equation}
\begin{aligned}
K_B
&:=
\frac{\partial\mathcal F}{\partial z}
=
\begin{bmatrix}
H & A^\top\\
A & 0
\end{bmatrix},\\
G_B
&:=
\frac{\partial\mathcal F}{\partial\theta}
=
\begin{bmatrix}
B\\
C
\end{bmatrix}.
\end{aligned}
\label{eq:target_derivatives}
\end{equation}
Here, $H$ is the Hessian of the fixed-barrier Lagrangian,
$A=\partial q/\partial\xi$, and $B$ and $C$ contain the corresponding
derivatives with respect to $\theta$. All quantities are evaluated at
$(\xi^\star(\theta),\lambda^\star(\theta),\theta)$ with $\tau$ held
fixed. The subscript $B$ distinguishes this derivative-defining target
operator from the cached forward operator retained by the solver.

The implicit function theorem \cite{dontchev2014implicit} applied to
$\mathcal F(z^\star(\theta),\theta)=0$ gives
\begin{equation}
K_B
\frac{\partial z^\star}{\partial\theta}
=
-G_B.
\label{eq:solution_sensitivity}
\end{equation}
Let $\mathcal L(\xi,\theta)$ be a continuously differentiable outer
loss. Because $\mathcal L$ does not depend on the equality
multipliers, define the loss-induced adjoint right-hand side
\begin{equation}
g:=
\operatorname{col}
\bigl(\nabla_\xi\mathcal L,0\bigr).
\label{eq:loss_seed}
\end{equation}
The prescribed target adjoint is the solution of
\begin{equation}
K_B^\top a^\star=g.
\label{eq:adjoint_system}
\end{equation}
The resulting outer gradient is
\begin{equation}
\frac{\mathrm d\mathcal L}{\mathrm d\theta}
=
\nabla_\theta\mathcal L-G_B^\top a^\star.
\label{eq:loss_gradient}
\end{equation}
The backward task is therefore to compute the adjoint $a^\star$ defined
by \eqref{eq:adjoint_system} efficiently while preserving $K_B$ as the
derivative-defining operator.

\section{Methodology}
\label{sec:method}

The method separates the solver-owned cache from the derivative-defining
operator. The forward pass retains matrix-dependent generalized Riccati
information; the backward pass exposes it only as an approximate inverse
action and accepts an adjoint only under the prescribed target residual.

\subsection{Forward Pass}
\label{sec:forward_pass}

The forward pass computes the nominal solution by applying a
primal--dual interior-point method to the smooth fixed-barrier KKT
conditions defined in Section~\ref{sec:problem}. Under the standard
regularity conditions of interior-point methods, solutions of the
barrier subproblems approach a KKT point of the original
inequality-constrained problem as $\tau$ decreases
\cite{wachter2006ipopt}. For finite $\tau$, however, the barrier
problem defines a smooth approximation. Throughout this work, both the
nominal solution and the prescribed backward derivative refer to this
same fixed-$\tau$ solution map.

At iteration $i$, the forward solver computes the step
$\Delta\bar z_i$ from the stage-structured reduced Newton system
\begin{equation}
K_i\Delta\bar z_i=b_i,
\label{eq:forward_reduced_system}
\end{equation}
where $K_i$ is the reduced Newton operator obtained from the current
smooth KKT conditions, $\Delta\bar z_i$ contains the solver's internal
primal--dual increments, and $b_i$ is the iteration-dependent
right-hand side. The matrix-dependent part of this system is
factorized through a generalized Riccati recursion.

The generalized Riccati recursion separates matrix-dependent
eliminations from the vector recursions induced by $b_i$
\cite{vanroye2024generalization}. Its reusable state comprises
stagewise Riccati factors, equality-constraint eliminations,
coordinate transformations, and primal--dual recovery maps, rather
than a single canonical matrix factorization. Right-hand-side-dependent
vector recursions are therefore recomputed for each new right-hand side.

Let $i_F$ denote the iteration of the last successful forward
factorization. At termination, we retain its matrix-dependent cache,
\begin{equation}
\mathcal C_F:=\mathcal C_{i_F}.
\label{eq:retained_forward_cache}
\end{equation}
The converged nominal solution $z^\star$ satisfies the prescribed
fixed-barrier optimality map $\mathcal F(z^\star,\theta)=0$ to the
stated forward tolerance. Accordingly, the backward target $K_B$ is
the Jacobian of the same smooth optimality map evaluated at
$z^\star$.

Nevertheless, the retained cache $\mathcal C_F$ need not represent
$K_B$. It may originate from an earlier iterate, a different barrier
parameter, or a system containing regularization, curvature
approximations, and iterate-dependent inequality eliminations.
Therefore, the forward solution and backward derivative are aligned
through the same fixed-barrier optimality map, while the retained
linear-algebraic cache may still be mismatched with the prescribed
backward operator. The backward pass exploits this cache only as
approximate inverse information and never uses it to redefine the
derivative.

\subsection{Backward Pass}
\label{sec:backward_pass}

The backward pass must reuse the retained cache without allowing it to
define the derivative. To enforce this separation, we expose the cache
through the fixed inverse action
\begin{equation}
\mathcal S_F(v)
:=
\operatorname{CachedSolve}(\mathcal C_F,v).
\label{eq:cached_solve_action}
\end{equation}

The retained successful factorization is assumed to define a fixed
nonsingular linear action on the primal--dual space. Concretely,
\texttt{CachedSolve} maps the supplied target-coordinate right-hand
side through the cached stagewise transformations and eliminations,
performs the fixed vector recursions, and applies the cached recovery
maps to return a primal--dual vector in the ordering of $z$. No
explicit forward or target KKT matrix is formed, and all
matrix-dependent cached quantities remain fixed throughout the backward computation.
Although the canonical KKT operator in \eqref{eq:target_derivatives} is
symmetric, we retain the transpose in the adjoint equations to make their
reverse-mode origin explicit; it does not require a distinct transpose
factorization.

With this fixed action in place, a nonzero loss seed $g$ first produces
the cached adjoint candidate
\begin{equation}
a_0:=\mathcal S_F(g).
\label{eq:cached_adjoint_candidate}
\end{equation}
Rather than treating this candidate as the target adjoint, we evaluate
it using the prescribed target operator:
\begin{equation}
r_0:=g-K_B^\top a_0,
\qquad
\rho_0:=\frac{\|r_0\|_2}{\|g\|_2}.
\label{eq:initial_target_residual}
\end{equation}
It is accepted only when
$\rho_0\leq\varepsilon$, where $\varepsilon>0$ is the target-residual
tolerance.

Before constructing a correction for an rejected candidate, we first
characterize the mismatch responsible for its target residual. Define
the cache-action error operator
\begin{equation}
E_F:=I-K_B^\top\mathcal S_F.
\label{eq:normalized_mismatch_operator}
\end{equation}
This definition uses only algorithm-available actions and does not
introduce an implicit matrix representation of the cache.

\begin{proposition}[Cache-action mismatch characterization]
\label{prop:mismatch_dependent_correction}
Assume that $K_B$ is nonsingular and $\mathcal S_F$ is a fixed
nonsingular linear action. The cached candidate
satisfies
\begin{equation}
r_0=E_Fg.
\label{eq:directional_candidate_residual}
\end{equation}
Consequently,
\begin{equation}
a_0=a^\star
\quad\Longleftrightarrow\quad
E_Fg=0.
\label{eq:direction_wise_exact_reuse}
\end{equation}
Moreover, the right-preconditioned target operator has the form
\begin{equation}
K_B^\top\mathcal S_F=I-E_F.
\label{eq:preconditioned_mismatch_form}
\end{equation}
In exact arithmetic, let $r_j$ denote the target residual after $j$
iterations of unrestarted GMRES applied from a zero correction. Then
\begin{equation}
\|r_j\|_2
\leq
\|E_F^j r_0\|_2
\leq
\|E_F\|_2^{j+1}\|g\|_2.
\label{eq:mismatch_dependent_residual_bound}
\end{equation}
\end{proposition}
\begin{proof}
By definition of $E_F$ and $a_0=\mathcal S_F(g)$,
\[
r_0
=
g-K_B^\top\mathcal S_F(g)
=
E_Fg.
\]
Nonsingularity of $K_B$ implies that $r_0=0$ if and only if
$a_0$ is the unique solution of $K_B^\top a=g$. Furthermore,
\[
K_B^\top\mathcal S_F
=
I-E_F.
\]
GMRES minimizes its residual over polynomials
$p\in\Pi_j$ satisfying $p(0)=1$. Selecting
$p_j(t)=(1-t)^j$ gives
$p_j(I-E_F)=E_F^j$, and hence
\[
\|r_j\|_2
\leq
\|E_F^j r_0\|_2.
\]
The second inequality follows from $r_0=E_Fg$.
\end{proof}

\begin{figure}[!t]
\vspace{2mm}
\begin{minipage}{\columnwidth}
\small
\raggedright
\refstepcounter{dfalgorithm}
\label{alg:factor_guided_backward}

\renewcommand{\algorithmicrequire}{\textbf{Input:}}
\renewcommand{\algorithmicensure}{\textbf{Output:}}

\hrule height 0.6pt
\vspace{3pt}
\noindent
\textbf{Algorithm~\thedfalgorithm}\quad Backward Pass\par
\vspace{3pt}
\hrule height 0.3pt
\vspace{2pt}

\begin{algorithmic}[1]
\setlength{\itemsep}{0.5pt}

\REQUIRE Target action $v\mapsto K_B^\top v$,
$G_B$, $g$, $\nabla_\theta\mathcal L$,
cache action $\mathcal S_F$, and tolerance $\varepsilon$
\ENSURE Gradient estimate $\widehat d_\theta$

\IF{$\|g\|_2=0$}
    \RETURN $\nabla_\theta\mathcal L$
\ENDIF

\STATE $a_0\gets\mathcal S_F(g)$
\STATE $r_0\gets g-K_B^\top a_0$,
$\rho_0\gets\|r_0\|_2/\|g\|_2$

\IF{$\rho_0\leq\varepsilon$}
    \STATE $\widehat a\gets a_0$
\ELSE
    \STATE Initialize GMRES for
    $K_B^\top\mathcal S_F(y)=r_0$ with $y_0=0$
    \STATE $j\gets0$
    \WHILE{$\rho_j>\varepsilon$}
        \STATE $j\gets j+1$ and advance GMRES to $y_j$
        \STATE $a_j\gets a_0+\mathcal S_F(y_j)$
        \STATE $r_j\gets g-K_B^\top a_j$,
        $\rho_j\gets\|r_j\|_2/\|g\|_2$
    \ENDWHILE
    \STATE $\widehat a\gets a_j$
\ENDIF

\STATE $\widehat d_\theta\gets
\nabla_\theta\mathcal L-G_B^\top\widehat a$
\RETURN $\widehat d_\theta$

\end{algorithmic}

\vspace{2pt}
\hrule height 0.6pt
\end{minipage}
\end{figure}

Proposition~\ref{prop:mismatch_dependent_correction} links the two uses
of the cache through the same error operator. First, $E_Fg$ determines
whether direct reuse is exact for the current loss direction. Second,
$E_F$ determines the form of the right-preconditioned target operator
and provides one admissible GMRES residual bound. This is a structural
characterization rather than a general GMRES convergence theorem:
$\|E_F\|_2<1$ is sufficient, but not necessary, for effective
correction. In particular, exact direct reuse is seed-dependent; a
mismatched cache can still return the exact adjoint for a specific loss
direction when $E_Fg=0$.

We now use this characterization to correct a candidate that fails the
target-residual test. When $\rho_0>\varepsilon$, write the corrected
adjoint as $a=a_0+\delta a$. Substitution into the target equation gives
\begin{equation}
K_B^\top\delta a=r_0.
\label{eq:adjoint_correction_system}
\end{equation}
We parameterize the correction using the cached inverse action,
\begin{equation}
\delta a=\mathcal S_F(y),
\label{eq:right_preconditioned_variable}
\end{equation}
which yields
\begin{equation}
K_B^\top\mathcal S_F(y)=r_0,
\qquad
a=a_0+\mathcal S_F(y).
\label{eq:right_preconditioned_correction}
\end{equation}

Because $\mathcal S_F$ is nonsingular, this change of variables is
bijective and therefore preserves the original correction equation.
Each Krylov product evaluates
$v\mapsto K_B^\top\mathcal S_F(v)$: the target supplies the matrix
action and the cache supplies fixed right preconditioning.
Although $K_B$ is symmetric in the canonical coordinates of
\eqref{eq:target_derivatives}, the right-preconditioned operator
$K_B^\top\mathcal S_F$ is generally nonsymmetric. We therefore use GMRES \cite{saad1986gmres} rather than a
symmetry-dependent Krylov method. In the
implementation, the target action is evaluated directly from packed
stage Hessian and dynamics-Jacobian blocks. Thus, ``matrix-free'' refers
to the global target KKT matrix: neither that matrix nor its
factorization is formed.

Although GMRES operates in the preconditioned variable $y$, every
candidate is accepted according to its residual under the original
target operator:
\begin{equation}
r_j:=g-K_B^\top a_j,
\qquad
\rho_j:=\frac{\|r_j\|_2}{\|g\|_2}.
\label{eq:target_adjoint_residual}
\end{equation}
The returned adjoint $\widehat a$ must satisfy
$\rho_j\leq\varepsilon$.

The certified adjoint gives the outer-gradient estimate
$\widehat d_\theta
=\nabla_\theta\mathcal L-G_B^\top\widehat a$.
To relate the residual certificate to the gradient error, define
$r:=g-K_B^\top\widehat a$ and
$d_\theta^\star
:=\nabla_\theta\mathcal L-G_B^\top a^\star$.
Since $r=K_B^\top(a^\star-\widehat a)$,
\begin{equation}
\widehat d_\theta-d_\theta^\star
=
G_B^\top K_B^{-\top}r.
\label{eq:gradient_error_identity}
\end{equation}
Consequently,
\begin{equation}
\|\widehat d_\theta-d_\theta^\star\|_2
\leq
\varepsilon
\|G_B^\top K_B^{-\top}\|_2
\|g\|_2.
\label{eq:certified_gradient_error_bound}
\end{equation}
Thus, the target residual certifies satisfaction of the prescribed
adjoint equation, while its implication for the outer gradient also
depends on the conditioning of $K_B$ and the coupling through $G_B$.

Algorithm~\ref{alg:factor_guided_backward} summarizes candidate
construction, target-residual testing, and cache-guided correction.

Candidate construction requires one cache application and one target
product. With $m$ correction iterations, the backward pass uses
$1+m$ cache applications and at most $1+2m$ target products, in
addition to the Arnoldi updates. It avoids refactorizing $K_B$, although
the benefit remains mismatch- and conditioning-dependent.

\section{Experiments}
\label{sec:experiments}

The experiments evaluate the computational benefit of factorization reuse,
its effectiveness in open- and closed-loop parameter learning, and the
transfer of the learned controller to a physical vehicle. The forward solver uses a tolerance of $10^{-8}$, and every backward
computation reaches a target relative-residual tolerance of $10^{-6}$. We report backward time for the factor-reuse
ablation and forward and backward solver times for the learning
experiments. All computational experiments run single-threaded on an
Intel Core i5-14600KF under Windows~11.

\subsection{Factorization-Reuse Ablation}
\label{sec:reuse_ablation}

This ablation evaluates whether the structured information retained by
the forward pass reduces backward computation relative to reconstructing
the target factorization or discarding the cache entirely.

\begin{figure}[!t]
	\centering
	\makebox[\columnwidth][l]{%
		\hspace*{-3mm}%
		\includegraphics[width=\columnwidth]{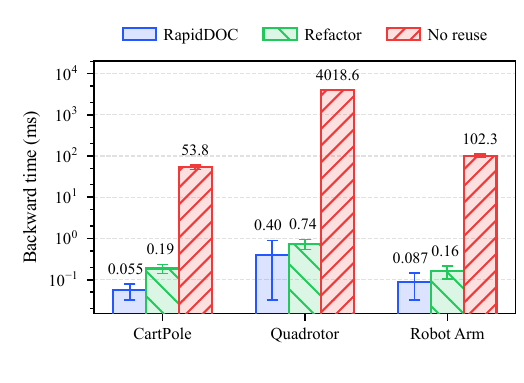}%
	}
	\caption{Backward time in the factorization-reuse ablation. Error bars show
		one standard deviation over 32 instances; the vertical axis is
		logarithmic.}
	\label{fig:reuse_ablation}
\end{figure}

We compare three backward strategies. RapidDOC uses the retained cache for
candidate construction and, when required, fixed right preconditioning.
Refactor discards the cache and constructs a new generalized Riccati
factorization for the target system. No reuse applies zero-start,
unpreconditioned GMRES to the target system. All variants use the same
forward result, target operator, adjoint right-hand side, and stopping
criterion.

We evaluate 32 instances for each of CartPole, Quadrotor, and Robot
Arm. Their horizons are 30, 50, and 35, respectively, with corresponding
target-adjoint dimensions of 278, 1424, and 358. The reported times
cover only the backward computation, including target factorization for
Refactor and cache application and correction for RapidDOC.

As shown in Fig.~\ref{fig:reuse_ablation}, RapidDOC achieves the lowest
backward time on all three tasks. Its advantage over Refactor shows
that the retained forward structure avoids reconstructing and
factorizing the target system. Its advantage over No reuse further
shows that the cache provides useful inverse information for the target
backward computation. The ablation therefore confirms the computational
value of retaining the forward factorization under a common target
operator and stopping criterion.

\begin{figure*}[!t]
\centering
\includegraphics[width=0.94\textwidth]
{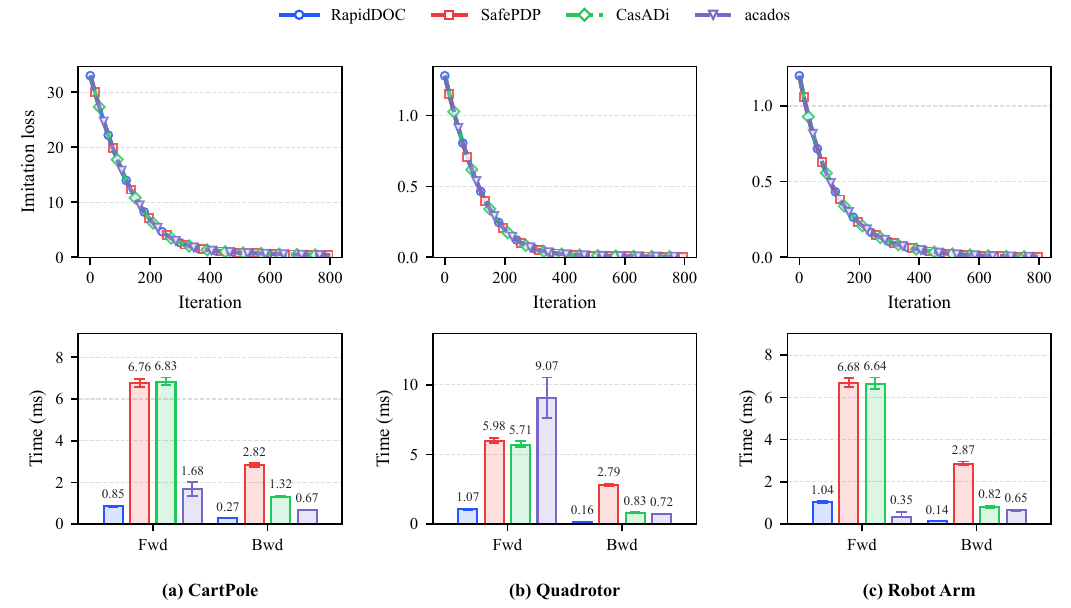}
\caption{Open-loop parameter learning. Top: trajectory-matching loss.
Bottom: forward and backward solver time; whiskers show interquartile
ranges.}
\label{fig:imitation_learning}
\end{figure*}

\begin{figure*}[!t]
\centering
\includegraphics[width=0.96\textwidth]
{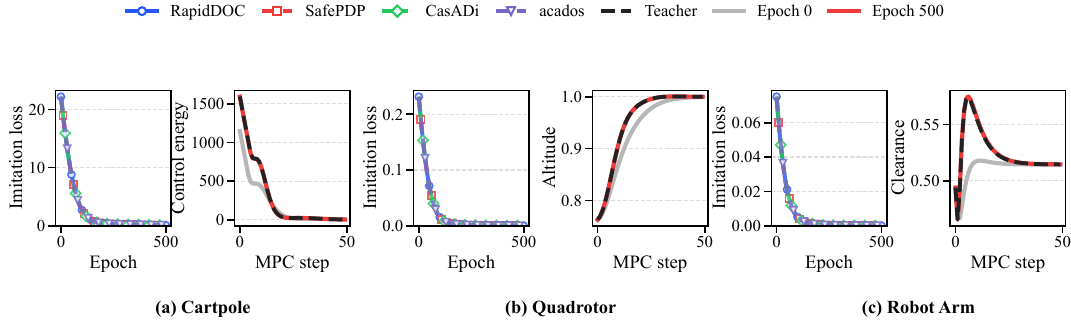}
\caption{Closed-loop parameter learning. Left: training loss. Right:
teacher and learned trajectories before and after training.}
\label{fig:closed_loop_imitation_learning}
\end{figure*}

\subsection{OCP Parameter Learning from Demonstrations}
\label{sec:imitation_learning}

We next evaluate factorization reuse in complete OCP parameter-learning
pipelines. For the $j$th initial condition $\zeta_j$, let
$\xi_j^\star(\theta):=\xi^\star(\theta;\zeta_j)$ denote the trajectory
generated by the learner and $\xi_j^{\mathrm{ref}}$ the corresponding
demonstration. The objective parameters are learned by minimizing
\begin{equation}
\mathcal L_{\mathrm{demo}}(\theta)
=
\frac{1}{2N_{\mathrm{demo}}}
\sum_{j=1}^{N_{\mathrm{demo}}}
\left\|
\xi_j^\star(\theta)-\xi_j^{\mathrm{ref}}
\right\|_W^2,
\label{eq:imitation_loss}
\end{equation}
where $\|e\|_W^2=e^\top W e$ and $W$ accounts for component scales
and horizon normalization.

Teacher and learner share the dynamics and constraints but use
different objective parameters. We consider CartPole, Quadrotor, and
Robot Arm tasks with five, six, and five log-parameterized objective
weights, respectively. All methods use the same demonstrations,
initialization, loss, and Adam updates \cite{kingma2015adam}. We compare RapidDOC with SafePDP \cite{jin2021safepdp}, a
CasADi--IPOPT reverse-mode pipeline \cite{andersson2019casadi},
and acados \cite{verschueren2022acados}.

\subsubsection{Open-Loop Learning}

The open-loop experiment uses 32 training and 16 validation
demonstrations. Each OCP has $N=40$ and a sampling interval of
$0.05$~s. All methods start from the same seed-2 initialization and
are trained for 800 iterations with a learning rate of
$2\times10^{-3}$.

Figure~\ref{fig:imitation_learning} shows closely aligned learning
trajectories across methods, indicating that they identify comparable
objectives from the demonstrations. RapidDOC also achieves the lowest
backward time on all three tasks. By reusing the structured factorizations
produced by the current forward computation, the backward pass avoids
constructing a separate factorization and reduces the work required
for correction.

The comparatively high acados time on Quadrotor is dominated by its
forward computation, consistent with the higher cost of exact-Hessian SQP on
this larger and more strongly coupled problem. Overall, the results
show that factorization reuse reduces the backward cost while preserving the
observed open-loop learning behavior.

\subsubsection{Closed-Loop Learning}

The closed-loop experiment applies the same learning procedure through
50 state-coupled MPC steps. Each OCP uses $N=40$ and a sampling interval
of $0.05$~s. All methods are trained from the same initialization for
500 epochs with a learning rate of $5\times10^{-3}$.

Figure~\ref{fig:closed_loop_imitation_learning} shows that the learned
trajectories approach their demonstrations and that the learning curves
remain closely aligned across methods. The behavior observed in the
open-loop experiment therefore persists when successive OCPs are
coupled through closed-loop state propagation.

Table~\ref{tab:closed_loop_performance} shows that RapidDOC achieves the
lowest total solver time on all three tasks. The forward factors are
updated and reused at every MPC step, so the computational benefit
accumulates over the rollout. These results show that factorization reuse
remains effective as the state, trajectory, and optimization problem
evolve during closed-loop learning.

\begin{table}[!t]
\centering
\caption{Median forward, backward, and total solver time per MPC step
during closed-loop learning (ms).}
\label{tab:closed_loop_performance}
\small
\renewcommand{\arraystretch}{1.08}
\setlength{\tabcolsep}{3pt}
\begin{tabular*}{\columnwidth}{@{\extracolsep{\fill}}llccc@{}}
\toprule
Task & Method & Fwd. & Bwd. & Total \\
\midrule
CartPole
& RapidDOC    & \textbf{0.77} & \textbf{0.13} & \textbf{0.91} \\
& SafePDP & 6.18 & 2.78 & 8.97 \\
& CasADi  & 6.06 & 2.51 & 8.59 \\
& acados  & 2.60 & 2.79 & 5.40 \\
\addlinespace
Quadrotor
& RapidDOC    & \textbf{1.07} & \textbf{0.79} & \textbf{1.86} \\
& SafePDP & 6.02 & 2.79 & 8.80 \\
& CasADi  & 5.58 & 2.09 & 7.69 \\
& acados  & 1.40 & 3.96 & 5.34 \\
\addlinespace
Robot Arm
& RapidDOC    & 0.92 & \textbf{0.63} & \textbf{1.60} \\
& SafePDP & 5.96 & 2.83 & 8.81 \\
& CasADi  & 5.62 & 2.14 & 7.76 \\
& acados  & \textbf{0.27} & 2.42 & 2.69 \\
\bottomrule
\end{tabular*}
\end{table}

\begin{figure*}[!t]
\centering
\includegraphics[width=\textwidth]
{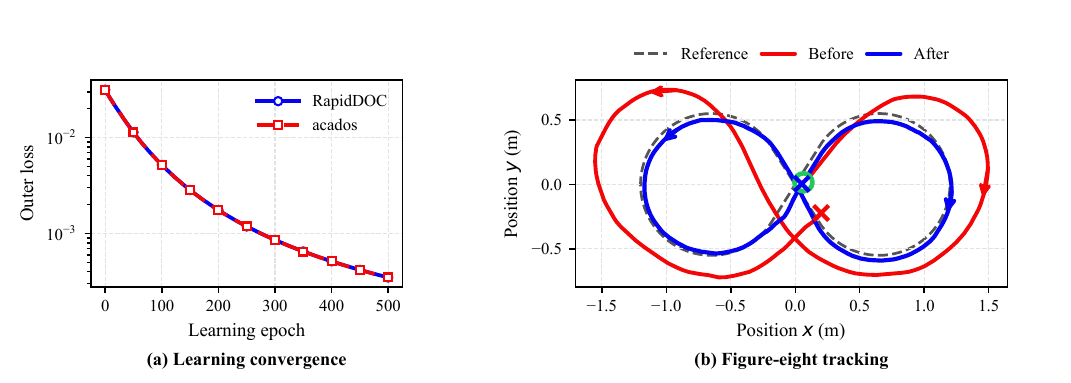}
\caption{Figure-eight imitation learning: (a) training losses of RapidDOC and acados; (b) trajectories before and
after learning.}
\label{fig:figure8_learning}
\end{figure*}

\begin{figure*}[!t]
\centering
\includegraphics[width=\textwidth]
{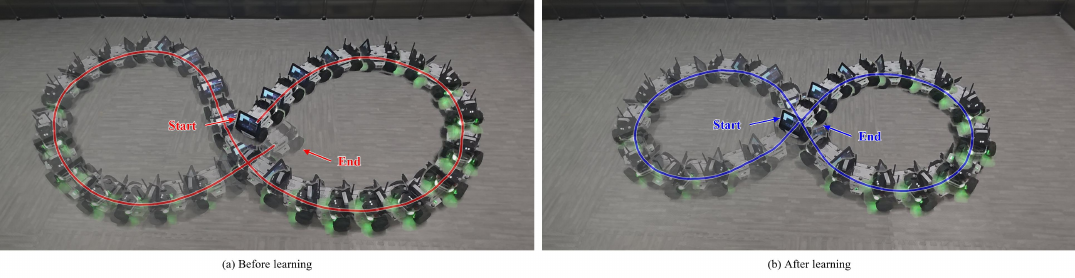}
\caption{Physical figure-eight tracking before and after learning over
the same $70.5$~s interval.}
\label{fig:physical_vehicle}
\end{figure*}

\subsection{Physical Vehicle Demonstration}
\label{sec:physical_vehicle}

To evaluate the potential of the proposed method for real-world
deployment, we consider a figure-eight imitation-learning task in which
an initially poor tracking controller is improved using a demonstrated
trajectory. The controller employs a five-state kinematic bicycle model
comprising planar position, yaw, speed, and reference phase, with
acceleration and steering as inputs. Four log-parameterized objective
weights penalizing position, heading, speed, and steering are optimized
to match the demonstrated trajectory. The MPC horizon is $N=25$, the
sampling interval is $0.1$~s, and the fixed barrier coefficient is
$\tau=10^{-3}$. Training is performed in simulation for 500 epochs using
707-step closed-loop rollouts.

Figure~\ref{fig:figure8_learning} reports the learning curves and the
tracking behavior obtained before and after training. We compare RapidDOC
with acados using the same demonstration and parameter initialization.
The two methods produce nearly identical decreases in the imitation
loss, indicating that the gradients computed by RapidDOC lead to a learning
outcome consistent with the reference solver. The trajectory comparison
further shows that the controller obtained from the initial parameters
deviates substantially from the demonstrated figure-eight path, whereas
the learned controller follows the reference much more closely. Thus,
the benefit of the computed derivatives is reflected not only in the
loss value but also in the resulting control behavior.

We next evaluate the computational cost of both methods on the onboard
computing platform. Table~\ref{tab:vehicle_learning_timing} reports the
median forward, backward, and total solver-call times per MPC step during
figure-eight learning. Backward time includes the forward recomputation
required by closed-loop reverse accumulation. RapidDOC requires $1.45$~ms
for the forward pass and $1.73$~ms for the backward pass, compared with
$2.08$~ms and $2.36$~ms for acados. Consequently, the median total
solver-call time is reduced from $4.44$ to $3.17$~ms per MPC step. This
result is consistent with the computational behavior observed in the
preceding experiments and shows that the proposed forward and backward
computations can be executed efficiently on the onboard processor.

\begin{table}[!t]
\centering
\caption{Median solver time per MPC step on the onboard platform (ms).}
\label{tab:vehicle_learning_timing}
\small
\renewcommand{\arraystretch}{1.08}
\setlength{\tabcolsep}{5pt}
\begin{tabular*}{\columnwidth}
{@{\extracolsep{\fill}}lccc@{}}
\toprule
Method & Fwd. & Bwd. & Total \\
\midrule
RapidDOC   & \textbf{1.45} & \textbf{1.73} & \textbf{3.17} \\
acados & 2.08 & 2.36 & 4.44 \\
\bottomrule
\end{tabular*}
\end{table}

Finally, the initial and learned controllers are deployed in the
onboard ROS system without further parameter updates, with the first
MPC input applied at $10$~Hz. We conduct three complete runs before
learning and two after learning. The sequences in
Fig.~\ref{fig:physical_vehicle} use the same camera viewpoint and scale,
without zooming or rescaling; the corresponding planar trajectories are
shown in Fig.~\ref{fig:figure8_learning}(b). The learned controller
follows the reference more closely, reducing the median position RMSE
from $22.75$~cm to $6.09$~cm. Given the limited number of runs, this
experiment demonstrates hardware deployment potential rather than
statistical superiority.

\section{Conclusion}

We presented \textbf{RapidDOC}, a factorization-reuse method for
differentiable OCPs under forward--backward operator mismatch.
\textbf{RapidDOC} uses retained generalized Riccati factors to initialize
and precondition the backward solve, while the prescribed target
operator remains the sole definition of the derivative.
A three-way ablation demonstrates the computational benefit of
factorization reuse. Open-loop, closed-loop, and physical-vehicle
experiments further show consistent parameter learning and the
potential for hardware deployment. This study considers a
fixed-barrier formulation and naturally occurring mismatches;
future work will address active-set changes and adaptive reuse
in embedded online differentiation.
\begingroup
\hyphenpenalty=10000
\exhyphenpenalty=10000
\emergencystretch=2em

\bibliographystyle{IEEEtran}
\bibliography{references}
\endgroup

\end{document}